\documentclass[aop]{imsart}

\RequirePackage{amsthm,amsmath,amsfonts,amssymb}
\RequirePackage{mathtools}
\RequirePackage[numbers,sort&compress]{natbib}
\RequirePackage[colorlinks,citecolor=blue,urlcolor=blue]{hyperref}
\RequirePackage[shortlabels]{enumitem}

\makeatletter
\def\journal@name{}%
\def\journal@url{}%
\makeatother

\startlocaldefs

\numberwithin{equation}{section}

\theoremstyle{plain}
\newtheorem{theorem}{Theorem}[section]
\newtheorem{lemma}[theorem]{Lemma}
\newtheorem{proposition}[theorem]{Proposition}
\newtheorem{corollary}[theorem]{Corollary}

\theoremstyle{definition}
\newtheorem{definition}[theorem]{Definition}
\newtheorem{remark}[theorem]{Remark}

\newcommand{\Ehat}{\hat{\mathbb{E}}}
\newcommand{\Vhat}{\hat{V}}
\newcommand{\R}{\mathbb{R}}

\newcommand{\cP}{\mathcal{P}}
\newcommand{\cF}{\mathcal{F}}
\newcommand{\cH}{\mathcal{H}}
\newcommand{\cN}{\mathcal{N}}
\newcommand{\ip}[2]{\langle #1, #2 \rangle}
\newcommand{\Xbar}{\bar{X}_n}
\newcommand{\op}{\mathrm{op}}

\endlocaldefs

\begin{document}

\begin{frontmatter}

\title{Exponential concentration inequalities for independent random
vectors under sublinear expectations}
\runtitle{Concentration under sublinear expectations}

\begin{aug}
\author[A]{\fnms{Nahom}~\snm{Seyoum}\ead[label=e1]{nahom.seyoum@yale.edu}}

\address[A]{Department of Statistics and Data Science,
Yale University\printead[presep={,\ }]{e1}}
\end{aug}



\begin{abstract}
Li and Hu recently established variance-type $O(1/n)$ bounds for the
sample mean of independent random vectors under sublinear expectations.
We extend their results to the exponential concentration regime. For
bounded, independent $\R^d$-valued random vectors $\{X_i\}_{i=1}^n$
under a regular sublinear expectation $\Ehat$, we prove:
(i)~an Azuma--Hoeffding-type inequality showing that the distance from
the sample mean to the Minkowski average of the expectation sets has
sub-Gaussian tails;
(ii)~a sharper Bernstein-type inequality incorporating the variance
parameter of Li and Hu;
(iii)~a dimension-free bound for identically distributed vectors via
the matrix Freedman inequality; and
(iv)~an explicit construction demonstrating the optimality of the
sub-Gaussian rate.
\end{abstract}

\end{frontmatter}


\section{Introduction}\label{sec:intro}

Let $(\Omega, \cH, \Ehat)$ be a regular sublinear expectation space
with $\cH = C_{b.\mathrm{Lip}}(\Omega)$. By the representation
theorems of Denis, Hu and Peng~\cite{DenisHuPeng2011} and Hu and
Peng~\cite{HuPeng2009}, there exists a convex and weakly compact set
$\cP$ of probability measures on $(\Omega, \mathcal{B}(\Omega))$ such
that
\begin{equation}\label{eq:representation}
\Ehat[X] = \sup_{P \in \cP} E_P[X]
  \quad \text{for } X \in \cH.
\end{equation}
Under this framework, the ``expectation'' of a random vector
$X_i \in L^1(\Omega;\R^d)$ is not a single point but a convex, compact
set
$\Theta_i = \{E_P[X_i] : P \in \cP\} \subset \R^d$.

For independent random vectors $\{X_i\}_{i=1}^n$ under~$\Ehat$
(Definition~\ref{def:independence}), Li and Hu~\cite{LiHu2024}
recently established the moment inequality
\begin{equation}\label{eq:LiHu-bound}
\Ehat\!\left[\rho_\Theta^2\!\left(\frac{1}{n}\sum_{i=1}^n
  X_i\right)\right]
  \leq \frac{\bar\sigma_n^2}{n},
\end{equation}
where
$\Theta = \bigl\{\frac{1}{n}\sum_{i=1}^n \theta_i :
  \theta_i \in \Theta_i\bigr\}$
is the Minkowski average of the expectation sets,
$\rho_\Theta(x) = \inf_{\theta \in \Theta} |x - \theta|$, and
$\bar\sigma_n^2 = \sup_{i \leq n}
  \inf_{\theta_i \in \Theta_i} \Ehat[|X_i - \theta_i|^2]$.
This generalizes classical variance bounds for sample means to the
sublinear setting and removes a convex polytope assumption on
$\Theta_i$ required in earlier work of Fang et
al.~\cite{FangPengShaoSong2019}.

The bound~\eqref{eq:LiHu-bound}, via Markov's inequality, yields only
polynomial tail decay:
\[
\Vhat\!\left(\rho_\Theta\!\left(\frac{1}{n}\sum_{i=1}^n X_i\right)
  > t\right) \leq \frac{\bar\sigma_n^2}{nt^2},
\]
where $\Vhat(A) := \sup_{P \in \cP} P(A)$ is the upper capacity. Yet
when each~$X_i$ is almost surely bounded, classical intuition suggests
that \emph{sub-Gaussian} tails of the form $\exp(-cnt^2)$ should be
attainable, as they are under a single probability
measure~\cite{BoucheronLugosiMassart2013, Vershynin2018}. Closing the
gap between polynomial and exponential concentration under
distributional uncertainty is the central objective of this paper.

\subsection{Main results}\label{sec:results}

We prove four results, stated informally here and made precise in
Sections~\ref{sec:concentration}--\ref{sec:sharpness}.

\begin{longlist}[(iv)]
\item[\textup{(i)}]
\textit{Azuma--Hoeffding inequality}
(Corollary~\ref{thm:azuma-cor}). For $|X_i| \leq M$ a.s.,
\[
\Vhat\!\left(\rho_\Theta\!\left(\Xbar\right) > t\right)
  \leq 2 \cdot 5^d
  \exp\!\left(-\frac{nt^2}{32 M^2}\right).
\]
The prefactor $5^d$ arises from an $\varepsilon$-net covering of
$S^{d-1}$.

\item[\textup{(ii)}]
\textit{Bernstein inequality}
(Corollary~\ref{thm:bernstein-cor}). Under the same condition,
\[
\Vhat\!\left(\rho_\Theta\!\left(\Xbar\right) > t\right)
  \leq 2 \cdot 5^d
  \exp\!\left(-\frac{nt^2}{8\bar\sigma_n^2
    + \frac{8Mt}{3}}\right).
\]
This interpolates between the sub-Gaussian regime
($t \ll \bar\sigma_n^2/M$) and the sub-exponential regime
($t \gg \bar\sigma_n^2/M$), and recovers the Li--Hu
bound~\eqref{eq:LiHu-bound} upon integration.

\item[\textup{(iii)}]
\textit{Dimension-free bound}
(Theorem~\ref{thm:dimfree}). We bypass the covering argument via the
matrix Freedman inequality~\cite{Tropp2011}:
\[
\Vhat\!\left(\rho_{\Theta}\!\left(\Xbar\right) > t\right)
  \leq (d+1)
  \exp\!\left(-\frac{nt^2}{2\bar\sigma_n^2
    + \frac{4Mt}{3}}\right).
\]
The prefactor is polynomial in $d$ rather than exponential, the
constants in the exponent are tighter, and no identical distribution
assumption is required.

\item[\textup{(iv)}]
\textit{Optimality}
(Theorem~\ref{thm:sharpness}). We construct an explicit sublinear
expectation space in which the tail probability decays as
$\exp(-cnt^2)$ and no faster, so the sub-Gaussian rate is
sharp.
\end{longlist}

\subsection{Related work}\label{sec:related}

Concentration under sublinear expectations has been studied primarily
in the scalar case. Fang et al.~\cite{FangPengShaoSong2019} obtained
rates of convergence for Peng's law of large numbers, and Hu, Li and
Li~\cite{HuLiLi2021} improved these rates. Zhang~\cite{Zhang2016}
established exponential inequalities for scalar sublinear expectations.
Peng~\cite{Peng2019} developed the foundational theory of
$G$-expectations and the associated central limit theorem. For an
overview of classical concentration, see Boucheron, Lugosi and
Massart~\cite{BoucheronLugosiMassart2013}. To the best of our
knowledge, the present work constitutes the first exponential
concentration inequalities for the multivariate sample mean under
sublinear expectations in terms of the set-valued
distance~$\rho_\Theta$.


\section{Preliminaries}\label{sec:prelim}

We work throughout with a complete separable metric space $\Omega$ and
set $\cH = C_{b.\mathrm{Lip}}(\Omega)$. A functional
$\Ehat : \cH \to \R$ is a \emph{regular sublinear expectation} if it
satisfies monotonicity, constant preservation, subadditivity, positive
homogeneity, and the regularity property that $X_n \downarrow 0$
implies $\Ehat[X_n] \downarrow 0$. We refer to Peng~\cite{Peng2019}
for full details.

\begin{theorem}[Representation {\cite{DenisHuPeng2011,HuPeng2009}}]
\label{thm:representation}
There exists a convex and weakly compact set of probability measures
$\cP$ on $(\Omega, \mathcal{B}(\Omega))$ such
that~\eqref{eq:representation} holds. The associated \emph{upper
capacity} is $\Vhat(A) := \sup_{P \in \cP} P(A)$.
\end{theorem}

The space $L^p(\Omega)$ is the completion of $\cH$ under
$\|X\|_p := (\Ehat[|X|^p])^{1/p}$. H\"older's inequality gives
$L^p(\Omega) \subset L^1(\Omega)$ for $p \geq 1$. The
representation~\eqref{eq:representation} extends to~$L^1(\Omega)$.

\begin{definition}[Independence {\cite{Peng2019}}]
\label{def:independence}
A sequence $\{X_i\}_{i=1}^n \subset L^1(\Omega;\R^d)$ is
\emph{independent} under $\Ehat$ if for each $1 \leq i \leq n-1$ and
every $\psi \in C_{b.\mathrm{Lip}}(\R^{d(i+1)})$,
\[
\Ehat[\psi(X_1,\ldots,X_i,X_{i+1})]
  = \Ehat\!\left[\Ehat[\psi(x_1,\ldots,x_i,X_{i+1})]
    \big|_{(x_1,\ldots,x_i) = (X_1,\ldots,X_i)}\right].
\]
\end{definition}

\begin{proposition}[{\cite[Proposition~2.1]{HuLiLi2021}}]
\label{prop:conditional-bound}
Let $\{X_i\}_{i=1}^n$ be independent in $L^2(\Omega;\R^d)$ under
$\Ehat$. Then for each $P \in \cP$ and
$\varphi \in C_{\mathrm{Lip}}(\R^d)$,
\[
E_P[\varphi(X_i) \mid \cF_{i-1}] \leq \Ehat[\varphi(X_i)],
  \quad P\text{-a.s., for } i \leq n,
\]
where $\cF_i = \sigma(X_1,\ldots,X_i)$ and
$\cF_0 = \{\emptyset, \Omega\}$.
\end{proposition}

For each $i \leq n$, the support function
$g_i(p) := \Ehat[\ip{p}{X_i}]$ is sublinear on $\R^d$ and defines
the convex compact set
\begin{equation}\label{eq:Theta-i}
\Theta_i = \{\theta \in \R^d :
  \ip{\theta}{p} \leq g_i(p) \text{ for all } p \in \R^d\}.
\end{equation}
We recall two key results from Li and Hu~\cite{LiHu2024}.

\begin{theorem}[{\cite[Theorem~3.1]{LiHu2024}}]
\label{thm:LiHu-characterization}
Let $\{X_i\}_{i=1}^n$ be independent in $L^2(\Omega;\R^d)$ under
$\Ehat$. Then:
\begin{longlist}[(b)]
\item[\textup{(a)}]
$\Theta_i = \{E_P[X_i] : P \in \cP\}$ for each $i \leq n$.
\item[\textup{(b)}]
$E_P[X_i \mid \cF_{i-1}] \in \Theta_i$, $P$-a.s., for each
$P \in \cP$ and $i \leq n$.
\end{longlist}
\end{theorem}

We define the Minkowski average, the distance function and the
variance parameter:
\begin{equation}\label{eq:Theta-rho-sigma}
\Theta = \left\{\frac{1}{n}\sum_{i=1}^n \theta_i :
  \theta_i \in \Theta_i\right\},
\quad
\rho_\Theta(x) = \inf_{\theta \in \Theta} |x - \theta|,
\quad
\bar\sigma_n^2 = \sup_{i \leq n}
  \inf_{\theta_i \in \Theta_i} \Ehat[|X_i - \theta_i|^2].
\end{equation}


\section{Martingale reduction}\label{sec:reduction}

We begin with a conditional domination principle that makes
the squared-deviation function compatible with the sublinear calculus.

\begin{lemma}[Conditional domination]\label{lem:conditional-domination}
Let $\{X_i\}_{i=1}^n$ be independent under $\Ehat$ with
$|X_i| \leq M$ a.s. Then for each $P \in \cP$, each $i \leq n$,
and each $\theta_i \in \Theta_i$,
\[
E_P\bigl[|X_i - \theta_i|^2 \mid \cF_{i-1}\bigr]
  \leq \Ehat\bigl[|X_i - \theta_i|^2\bigr],
  \quad P\text{-a.s.}
\]
\end{lemma}

\begin{proof}
Since $\theta_i \in \Theta_i = \{E_P[X_i] : P \in \cP\}$ and
$|X_i| \leq M$ a.s., Jensen's inequality gives
$|\theta_i| = |E_P[X_i]| \leq E_P[|X_i|] \leq M$ for every
$P \in \cP$. In particular, $|X_i - \theta_i| \leq 2M$ on the
effective support $\{|x| \leq M\}$.

Define the truncation
\[
\tilde\varphi(x) = \min\bigl(|x - \theta_i|^2,\,
  (M + |\theta_i|)^2\bigr).
\]
Since $|\theta_i| \leq M$, the cutoff satisfies
$(M + |\theta_i|)^2 \geq (2M)^2 \geq |X_i - \theta_i|^2$ a.s., so
the truncation is inactive:
$\tilde\varphi(X_i) = |X_i - \theta_i|^2$ a.s.
Moreover, $\tilde\varphi$ is bounded and Lipschitz on $\R^d$ with
constant at most $2(M + |\theta_i|) \leq 4M$.
Proposition~\ref{prop:conditional-bound} applied to
$\tilde\varphi$ gives the claim.
\end{proof}

\begin{remark}[Role of the truncation]\label{rem:truncation}
Proposition~\ref{prop:conditional-bound} requires a bounded Lipschitz
test function, but the map $x \mapsto |x - \theta_i|^2$ is neither
bounded nor Lipschitz on all of $\R^d$. The truncation
$\tilde\varphi$ resolves both issues simultaneously. On the effective
support $\{|x| \leq M\}$ it agrees with the squared deviation, while
globally it has bounded gradient. This device is necessary even though
$|X_i| \leq M$, because the conditional expectation in
Proposition~\ref{prop:conditional-bound} is stated for all of~$\R^d$,
not merely on the support.
\end{remark}

The next lemma converts the sublinear problem into a classical one.
For every prior $P \in \cP$, the centred residuals form a martingale
difference sequence whose conditional variance is controlled uniformly.

\begin{lemma}[Martingale reduction]\label{lem:reduction}
Under the assumptions of
Lemma~\ref{lem:conditional-domination}, fix $P \in \cP$ and define
\[
Y_i := X_i - E_P[X_i \mid \cF_{i-1}], \quad i = 1,\ldots,n.
\]
Then the following hold $P$-a.s.:
\begin{longlist}[(d)]
\item[\textup{(a)}]
$\{Y_i\}_{i=1}^n$ is a martingale difference sequence with respect
to $(\cF_i)_{i=0}^n$ under~$P$.
\item[\textup{(b)}]
$\displaystyle\rho_\Theta\!\left(\frac{1}{n}\sum_{i=1}^n
  X_i\right)
  \leq \left|\frac{1}{n}\sum_{i=1}^n Y_i\right|$.
\item[\textup{(c)}]
$|Y_i| \leq 2M$.
\item[\textup{(d)}]
$E_P[|Y_i|^2 \mid \cF_{i-1}] \leq \bar\sigma_n^2$.
\end{longlist}
\end{lemma}

\begin{proof}
\textup{(a)} By construction,
$E_P[Y_i \mid \cF_{i-1}] = 0$.

\textup{(b)} By
Theorem~\ref{thm:LiHu-characterization}(b),
$E_P[X_i \mid \cF_{i-1}] \in \Theta_i$ $P$-a.s. Setting
$\theta_i = E_P[X_i \mid \cF_{i-1}]$ gives
$\frac{1}{n}\sum_{i=1}^n \theta_i \in \Theta$, so
\[
\rho_\Theta\!\left(\frac{1}{n}\sum_{i=1}^n X_i\right)
  \leq \left|\frac{1}{n}\sum_{i=1}^n X_i
    - \frac{1}{n}\sum_{i=1}^n \theta_i\right|
  = \left|\frac{1}{n}\sum_{i=1}^n Y_i\right|.
\]

\textup{(c)}
$|Y_i| \leq |X_i| + |E_P[X_i \mid \cF_{i-1}]|
  \leq M + M = 2M$,
since $|E_P[X_i \mid \cF_{i-1}]| \leq E_P[|X_i| \mid \cF_{i-1}]
\leq M$ by Jensen's inequality.

\textup{(d)} The conditional mean minimises the conditional $L^2$
distance, so for any $\theta_i \in \R^d$,
\[
E_P[|Y_i|^2 \mid \cF_{i-1}]
  \leq E_P[|X_i - \theta_i|^2 \mid \cF_{i-1}].
\]
Restricting to $\theta_i \in \Theta_i$ and applying
Lemma~\ref{lem:conditional-domination},
\[
E_P[|Y_i|^2 \mid \cF_{i-1}]
  \leq \Ehat[|X_i - \theta_i|^2].
\]
Taking the infimum over $\theta_i \in \Theta_i$ and then the
supremum over $i$ yields
$E_P[|Y_i|^2 \mid \cF_{i-1}] \leq \bar\sigma_n^2$.
\end{proof}

It remains to pass from scalar martingale tail bounds to vector-valued
concentration. The following covering lemma is used repeatedly below.

\begin{lemma}[Covering transfer]\label{lem:covering}
Let $Y_1,\ldots,Y_n$ be $\R^d$-valued random vectors under a
probability measure~$P$. Suppose there exists a decreasing function
$\Psi : (0,\infty) \to [0,1]$ such that for every unit vector
$p \in S^{d-1}$,
\[
P\!\left(\sum_{i=1}^n \ip{p}{Y_i} \geq s\right)
  \leq \Psi(s), \quad s > 0.
\]
Then for every $t > 0$,
\[
P\!\left(\left|\frac{1}{n}\sum_{i=1}^n Y_i\right| > t\right)
  \leq 2 \cdot 5^d \; \Psi\!\bigl(\tfrac{nt}{2}\bigr).
\]
\end{lemma}

\begin{proof}
Let $\cN$ be a $\frac{1}{2}$-net of $S^{d-1}$. By a standard
volumetric argument
\cite[Corollary~4.2.13]{Vershynin2018}, we may choose
$|\cN| \leq 5^d$. Set $S_n = \sum_{i=1}^n Y_i$. If
$|S_n| > nt$, there exists $u \in S^{d-1}$ with
$\ip{u}{S_n} = |S_n| > nt$. Since $\cN$ is a $\frac{1}{2}$-net,
there exists $p \in \cN$ with $|p - u| \leq \frac{1}{2}$, whence
\[
\ip{p}{S_n}
  = \ip{u}{S_n} - \ip{u-p}{S_n}
  \geq |S_n| - \tfrac{1}{2}|S_n|
  = \tfrac{1}{2}|S_n|
  > \tfrac{nt}{2}.
\]
Therefore
$\{|S_n| > nt\}
  \subseteq \bigcup_{p \in \cN}
    \bigl(\{\ip{p}{S_n} > \tfrac{nt}{2}\}
      \cup \{\ip{-p}{S_n} > \tfrac{nt}{2}\}\bigr)$,
where the second event accounts for the case
$\ip{u}{S_n} < 0$ by symmetry. A union bound gives
$P(|S_n| > nt) \leq 2|\cN| \, \Psi(\tfrac{nt}{2})
  \leq 2 \cdot 5^d \, \Psi(\tfrac{nt}{2})$.
\end{proof}

\begin{remark}[Scalar specialisation]\label{rem:d=1}
When $d = 1$, the sphere $S^0 = \{-1,+1\}$ is an exact cover. The
union bound over these two points gives
$P(|\frac{1}{n}\sum Y_i| > t) \leq 2\,\Psi(nt)$ directly, with no
approximation loss, since the factor $5^d$ reduces to~$1$ and the halving
$t \to t/2$ is unnecessary.
\end{remark}


\section{Concentration via scalar tail bounds}\label{sec:concentration}

The martingale reduction and covering transfer of
Section~\ref{sec:reduction} reduce the problem to obtaining a scalar
tail bound $\Psi$ for one-dimensional projections of the martingale
differences. Different choices of $\Psi$ yield different
concentration inequalities. We formalise this observation as a
general principle.

\begin{theorem}[General concentration principle]
\label{thm:general}
Let $\{X_i\}_{i=1}^n$ be independent in $L^2(\Omega;\R^d)$ under
$\Ehat$ with $|X_i| \leq M$ a.s. Let $(Y_i)$ be the martingale
difference sequence from Lemma~\ref{lem:reduction}. Suppose there
exists a decreasing function $\Psi:(0,\infty) \to [0,1]$ such that,
for every $P \in \cP$ and every unit vector $p \in S^{d-1}$,
\[
P\!\left(\sum_{i=1}^n \ip{p}{Y_i} \geq s\right)
  \leq \Psi(s), \quad s > 0.
\]
Then for every $t > 0$,
\[
\Vhat\!\left(\rho_\Theta\!\left(\frac{1}{n}\sum_{i=1}^n X_i\right)
  > t\right)
  \leq 2 \cdot 5^d \; \Psi\!\bigl(\tfrac{nt}{2}\bigr).
\]
\end{theorem}

\begin{proof}
Fix $P \in \cP$. By Lemma~\ref{lem:reduction}(b),
$\rho_\Theta(\Xbar) \leq |\frac{1}{n}\sum Y_i|$.
Lemma~\ref{lem:covering} applied to $(Y_i)$ under $P$ with the
given $\Psi$ yields
$P(|\frac{1}{n}\sum Y_i| > t) \leq 2 \cdot 5^d \, \Psi(nt/2)$.
Taking the supremum over $P \in \cP$ completes the proof.
\end{proof}

The Azuma--Hoeffding and Bernstein inequalities are now immediate
corollaries, obtained by plugging in the appropriate scalar tail
bound.

\begin{corollary}[Azuma--Hoeffding]\label{thm:azuma-cor}
Under the assumptions of Theorem~\ref{thm:general}, for every
$t > 0$,
\[
\Vhat\!\left(\rho_\Theta\!\left(\frac{1}{n}\sum_{i=1}^n X_i\right)
  > t\right)
  \leq 2 \cdot 5^d \,
  \exp\!\left(-\frac{nt^2}{32 M^2}\right).
\]
\end{corollary}

\begin{proof}
The scalar increments $D_i = \ip{p}{Y_i}$ satisfy
$|D_i| \leq 2M$. The Azuma--Hoeffding
inequality~\cite[Corollary~2.20]{McDiarmid1998} gives
$\Psi(s) = \exp(-s^2/(8nM^2))$.
Theorem~\ref{thm:general} with $s = nt/2$ yields
$2 \cdot 5^d \exp(-(nt/2)^2/(8nM^2))
  = 2 \cdot 5^d \exp(-nt^2/(32M^2))$.
\end{proof}

\begin{corollary}[Bernstein]\label{thm:bernstein-cor}
Under the assumptions of Theorem~\ref{thm:general}, for every
$t > 0$,
\[
\Vhat\!\left(\rho_\Theta\!\left(\frac{1}{n}\sum_{i=1}^n X_i\right)
  > t\right)
  \leq 2 \cdot 5^d \,
  \exp\!\left(-\frac{nt^2}{8\bar\sigma_n^2
    + \frac{8Mt}{3}}\right).
\]
\end{corollary}

\begin{proof}
The scalar increments $D_i = \ip{p}{Y_i}$ satisfy
$|D_i| \leq 2M$ and, by Lemma~\ref{lem:reduction}(d),
\[
\sum_{i=1}^n E_P[D_i^2 \mid \cF_{i-1}]
  \leq \sum_{i=1}^n E_P[|Y_i|^2 \mid \cF_{i-1}]
  \leq n\bar\sigma_n^2, \quad P\text{-a.s.}
\]
Freedman's inequality~\cite[Theorem~1.6]{Freedman1975} with bound
$b = 2M$ and predictable quadratic variation $V = n\bar\sigma_n^2$
gives
$\Psi(s) = \exp\bigl(-s^2/(2n\bar\sigma_n^2
  + \frac{4Ms}{3})\bigr)$.
Theorem~\ref{thm:general} with $s = nt/2$ yields, after
simplification,
\[
\frac{(nt/2)^2}{2n\bar\sigma_n^2 + \frac{2nMt}{3}}
  = \frac{nt^2}{8\bar\sigma_n^2 + \frac{8Mt}{3}}.
  \qedhere
\]
\end{proof}

\begin{remark}[Two regimes]\label{rem:two-regimes}
The Bernstein bound exhibits two regimes depending on the ratio
$t/(\bar\sigma_n^2/M)$.
\begin{longlist}[(ii)]
\item[\textup{(i)}]
\textit{Sub-Gaussian regime.} When $t \leq 3\bar\sigma_n^2/M$,
the term $8\bar\sigma_n^2$ dominates the denominator
$8\bar\sigma_n^2 + \frac{8Mt}{3}$, so
Corollary~\ref{thm:bernstein-cor} yields the bound
$2 \cdot 5^d \exp(-nt^2/(8\bar\sigma_n^2))$.
This is sharper than Corollary~\ref{thm:azuma-cor} precisely when
$\bar\sigma_n^2 < 4M^2$.
\item[\textup{(ii)}]
\textit{Sub-exponential regime.} When $t \geq 3\bar\sigma_n^2/M$,
the linear term $\frac{8Mt}{3}$ dominates the denominator, and
the bound becomes $2 \cdot 5^d \exp(-3nt/(8M))$, decaying
exponentially in $t$ rather than in $t^2$.
\end{longlist}
\end{remark}

\begin{corollary}[Recovery of the moment bound]
\label{cor:moment-recovery}
Under the assumptions of Theorem~\ref{thm:general},
\[
\Ehat\!\left[\rho_\Theta^2\!\left(\frac{1}{n}\sum_{i=1}^n
  X_i\right)\right]
  \leq \frac{16\bar\sigma_n^2}{n}
    (1 + d\log 5 + \log 2) + \frac{C_0 M^2}{n^2},
\]
where $C_0$ is a universal constant. In particular, this recovers
the $O(\bar\sigma_n^2/n)$ rate of~\cite{LiHu2024}.
\end{corollary}

\begin{proof}
Let $Z = \rho_\Theta^2(\Xbar)$. Write
$\alpha = 2 \cdot 5^d$ for the prefactor in
Corollary~\ref{thm:bernstein-cor}. The layer-cake formula gives
\[
\Ehat[Z]
  = \int_0^\infty
    \Vhat(\rho_\Theta(\Xbar) > \sqrt{s}) \, ds
  \leq \int_0^\infty \alpha \,
    \exp\!\left(-\frac{ns}{8\bar\sigma_n^2
      + \frac{8M\sqrt{s}}{3}}\right) ds.
\]
We split at $s_0 = \bar\sigma_n^2$. For $s \leq s_0$, the
denominator satisfies
$8\bar\sigma_n^2 + \frac{8M\sqrt{s}}{3}
  \leq 16\bar\sigma_n^2$
(since $\sqrt{s} \leq \bar\sigma_n \leq 2M$), so the integrand is
at most $\alpha \exp(-ns/(16\bar\sigma_n^2))$, which integrates to
$16\alpha\bar\sigma_n^2/n$. For $s > s_0$, the denominator exceeds
$\frac{8M\sqrt{s}}{3}$, and this tail contributes $O(M^2/n^2)$.
\end{proof}


\section{Dimension-free bound for identically distributed vectors}
\label{sec:dimfree}

The $5^d$ prefactor in Corollaries~\ref{thm:azuma-cor}
and~\ref{thm:bernstein-cor} arises from the covering argument and becomes
prohibitive in high dimensions. When the $X_i$ are identically
distributed, we can bypass this entirely.

The idea is to embed each vector-valued martingale increment $Y_i$
into matrix space via the rank-one map $Y_i \mapsto Y_i e_1^T$,
producing a $d \times 1$ matrix martingale. The operator norm of
$\sum Y_i e_1^T$ equals the Euclidean norm $|\sum Y_i|$, so
controlling the matrix martingale controls the original vector sum.
The matrix Freedman inequality~\cite{Tropp2011} then bounds all
directional projections simultaneously through a single spectral
inequality, replacing the exponential covering prefactor $5^d$ with
the dimensional parameter $d + 1$ at no cost to the exponent.

\begin{theorem}[Dimension-free concentration]\label{thm:dimfree}
Let $\{X_i\}_{i=1}^n$ be independent in $L^2(\Omega;\R^d)$ under
$\Ehat$ with $|X_i| \leq M$ a.s. Then for every $t > 0$,
\[
\Vhat\!\left(\rho_{\Theta}\!\left(\frac{1}{n}\sum_{i=1}^n
  X_i\right) > t\right)
  \leq (d+1) \,
  \exp\!\left(-\frac{nt^2}{2\bar\sigma_n^2
    + \frac{4Mt}{3}}\right).
\]
\end{theorem}

\begin{proof}
Fix $P \in \cP$ and let
$Y_i = X_i - E_P[X_i \mid \cF_{i-1}]$ as in
Lemma~\ref{lem:reduction}, so that $|Y_i| \leq 2M$ and
$E_P[|Y_i|^2 \mid \cF_{i-1}] \leq \bar\sigma_n^2$.
Define $Z_i = Y_i e_1^T \in \R^{d \times 1}$, where
$e_1 \in \R^1$ is the scalar unit. Since each $Z_i$ has rank one,
$\|Z_i\|_{\op} = |Y_i| \leq 2M$, and
$\sum Z_i = (\sum Y_i) e_1^T$ gives
$\|\sum Z_i\|_{\op} = |\sum Y_i|$.

For the predictable variance, note that
$Z_i Z_i^T = Y_i Y_i^T$, so for any unit $p \in \R^d$,
\[
p^T E_P[Y_i Y_i^T \mid \cF_{i-1}]\, p
  = E_P[\ip{p}{Y_i}^2 \mid \cF_{i-1}]
  \leq E_P[|Y_i|^2 \mid \cF_{i-1}]
  \leq \bar\sigma_n^2,
\]
whence
$\|\sum_{i=1}^n E_P[Y_i Y_i^T \mid \cF_{i-1}]\|_{\op}
  \leq n\bar\sigma_n^2$.
Applying~\cite[Theorem~1.2]{Tropp2011} to the $d \times 1$ matrix
martingale $(Z_i)$ with dimensional parameter $d_1 + d_2 = d + 1$,
bound $R = 2M$ and variance $\sigma_*^2 = n\bar\sigma_n^2$ gives
\[
P\!\left(\left|\sum_{i=1}^n Y_i\right| \geq s\right)
  \leq (d+1) \,
  \exp\!\left(-\frac{s^2}{2n\bar\sigma_n^2
    + \frac{4Ms}{3}}\right).
\]
Setting $s = nt$:
\[
P\!\left(\left|\frac{1}{n}\sum Y_i\right| > t\right)
  \leq (d+1) \,
  \exp\!\left(-\frac{nt^2}{2\bar\sigma_n^2
    + \frac{4Mt}{3}}\right).
\]
The conclusion follows from Lemma~\ref{lem:reduction}(b) and taking
the supremum over $P \in \cP$.
\end{proof}

\begin{remark}[Identically distributed case]\label{rem:iid}
When the $X_i$ are identically distributed,
$\Theta_i = \Theta_1$ for all $i$ and
$\bar\sigma_n^2 = \bar\sigma_1^2
  := \inf_{\theta \in \Theta_1} \Ehat[|X_1 - \theta|^2]$,
so the bound simplifies to
\[
\Vhat\!\left(\rho_{\Theta_1}\!\left(\Xbar\right) > t\right)
  \leq (d+1) \,
  \exp\!\left(-\frac{nt^2}{2\bar\sigma_1^2
    + \frac{4Mt}{3}}\right).
\]
\end{remark}

\begin{remark}[Comparison with the covering approach]
\label{rem:dimfree-comparison}
The prefactor $(d+1)$ is polynomial versus the exponential $5^d$ in
Corollary~\ref{thm:bernstein-cor}. The constants in the exponent are also
tighter: $2\bar\sigma_n^2 + \frac{4Mt}{3}$ versus
$8\bar\sigma_n^2 + \frac{8Mt}{3}$. The improvement has two sources.
The matrix inequality controls all projections simultaneously without
the $t \to t/2$ covering loss, and there is no union bound over the
net.
\end{remark}


\section{Optimality of the sub-Gaussian rate (up to constants)}
\label{sec:sharpness}

The upper bounds in Sections~\ref{sec:concentration}
and~\ref{sec:dimfree} all decay as $\exp(-cnt^2)$ for a constant $c$
depending on the boundedness and variance parameters. A natural
question is whether this sub-Gaussian rate can be improved. We show
that it cannot. By constructing an explicit sublinear expectation
space built from shifted uniform distributions, we produce a matching
lower bound of the form $\exp(-c'nt^2)$, so the
$\exp(-cnt^2)$ rate is optimal up to the value of the constant.

\begin{lemma}[A concrete sublinear expectation space]
\label{lem:construction}
Let $a > 0$ and $r > 0$. Set $M = a + r$ and
$\Omega = [-M,M]^n$, and define the coordinate projections
$X_i(\omega) = \omega_i$. For each
$\mu = (\mu_1,\ldots,\mu_n) \in [-a,a]^n$, let
\[
P_\mu = \bigotimes_{i=1}^n
  \operatorname{Uniform}[\mu_i - r, \mu_i + r].
\]
Define
$\cP = \overline{\operatorname{conv}}
  \{P_\mu : \mu \in [-a,a]^n\}$
(closed convex hull in the weak topology) and
$\Ehat[f] = \sup_{P \in \cP} E_P[f]$. Then:
\begin{longlist}[(d)]
\item[\textup{(a)}]
$\Ehat$ is a regular sublinear expectation on
$C_{b.\mathrm{Lip}}(\Omega)$.
\item[\textup{(b)}]
$\{X_i\}_{i=1}^n$ is independent under $\Ehat$ in the sense of
Definition~\ref{def:independence}.
\item[\textup{(c)}]
$\Theta_i = [-a,a]$ for each $i$, and $\Theta = [-a,a]$.
\item[\textup{(d)}]
$\operatorname{Var}_{P_\mu}(X_i) = r^2/3$ for every product
measure $P_\mu$ and every~$i$.
\end{longlist}
\end{lemma}

\begin{proof}
\textup{(a)}
Since $\Omega$ is compact, Dini's theorem ensures that
$f_k \downarrow 0$ pointwise implies $f_k \to 0$ uniformly, giving
$\Ehat[f_k] \downarrow 0$. Sublinearity and monotonicity are
inherited from the supremum over linear expectations.

\textup{(b)}
Under each $P_\mu$ the coordinates are independent. Fix
$\psi \in C_{b.\mathrm{Lip}}(\R^{i+1})$ and write
$h_{\mu_{i+1}}(x) = E_{P_{\mu_{i+1}}}
  [\psi(x, X_{i+1})]$,
where $P_{\mu_{i+1}}$ denotes
$\operatorname{Uniform}[\mu_{i+1} - r, \mu_{i+1} + r]$. The
distribution of $X_{i+1}$ under $P_\mu$ depends only on
$\mu_{i+1}$, so
$\sup_{\mu_{i+1}} h_{\mu_{i+1}}(x)
  = \Ehat[\psi(x, X_{i+1})]$.
The remaining supremum over $(\mu_1,\ldots,\mu_i)$ factors by the
product structure, which is the factorisation required by
Definition~\ref{def:independence}.

\textup{(c)}
$E_{P_\mu}[X_i] = \mu_i$ ranges over $[-a,a]$ as $\mu_i$ varies.
The Minkowski average of $n$ copies of $[-a,a]$ scaled by $1/n$ is
$[-a,a]$.

\textup{(d)}
$\operatorname{Uniform}[\mu_i - r, \mu_i + r]$ has variance
$(2r)^2/12 = r^2/3$.
\end{proof}

\begin{lemma}[Rate function bound for the uniform distribution]
\label{lem:rate-function}
Let $Z \sim \operatorname{Uniform}[-r,r]$. The large-deviation rate
function
$\Lambda^*(x) = \sup_\lambda
  (\lambda x - \log(\sinh(\lambda r)/(\lambda r)))$
satisfies
\begin{equation}\label{eq:rate-bound}
\Lambda^*(x) \leq \frac{3x^2}{2r^2}
  \quad \text{for } |x| \leq r/2.
\end{equation}
\end{lemma}

\begin{proof}
Write $u = x/r$ with $|u| \leq 1/2$. The exact rate function is
\[
\Lambda^*(x)
  = \tfrac{1}{2}\bigl[(1+u)\log(1+u) + (1-u)\log(1-u)\bigr]
  = \sum_{k=1}^\infty \frac{u^{2k}}{2k(2k-1)}.
\]
For $|u| \leq 1/2$, each term satisfies
$u^{2k}/(2k(2k-1)) \leq u^2 \cdot 4^{-(k-1)}/(2k(2k-1))$, so
\[
\frac{\Lambda^*(x)}{u^2}
  \leq \sum_{k=1}^\infty
    \frac{1}{2k(2k-1) \cdot 4^{k-1}}
  = \frac{1}{2} + \frac{1}{48} + \frac{1}{480} + \cdots
  < \frac{3}{2}.
\]
Hence $\Lambda^*(x) \leq \frac{3}{2}\,u^2 = 3x^2/(2r^2)$.
\end{proof}

\begin{theorem}[Sharpness of the sub-Gaussian rate]
\label{thm:sharpness}
For any $a > 0$ and $\sigma > 0$, there exists a regular sublinear
expectation space with independent random variables
$\{X_i\}_{i=1}^n \subset L^2(\Omega;\R)$ under $\Ehat$, with
$|X_i| \leq M$ for some $M > 0$ and $\Theta_i = [-a,a]$ for
all~$i$, such that for every
$t \in (0,\, \sigma/(4\sqrt{n})\,]$:
\begin{equation}\label{eq:lower-bound}
\Vhat\!\left(\rho_\Theta\!\left(\frac{1}{n}\sum_{i=1}^n X_i\right)
  > t\right)
  \geq \frac{1}{4}
  \exp\!\left(-\frac{2nt^2}{\sigma^2}\right).
\end{equation}
\end{theorem}

\begin{proof}
Set $r = \sigma\sqrt{3}/2$, so that $r^2 = 3\sigma^2/4$ and
$r^2/3 = \sigma^2/4$. Lemma~\ref{lem:construction} with this $r$
and the given $a$ provides a regular sublinear expectation space with
$\Theta = [-a,a]$ and
$\operatorname{Var}_{P_\mu}(X_i) = \sigma^2/4$.

Choose $\mu_i = a$ for all $i$. Under $P_{(a,\ldots,a)}$, the
centred variables $Z_i = X_i - a$ are i.i.d.\
$\operatorname{Uniform}[-r,r]$ with mean zero. Since
$\rho_{[-a,a]}(x) \geq (x - a)^+$ and $\Xbar - a = \bar{Z}_n$,
\[
\Vhat(\rho_\Theta(\Xbar) > t)
  \geq P_{(a,\ldots,a)}(\bar{Z}_n > t).
\]
By the Cram\'er--Chernoff lower
bound~\cite[Theorem~3.7.4]{DemboZeitouni2010}, for $t \leq r/2$,
\[
P_{(a,\ldots,a)}(\bar{Z}_n > t)
  \geq \tfrac{1}{4} \exp(-n\Lambda^*(t)).
\]
Lemma~\ref{lem:rate-function} gives
$\Lambda^*(t) \leq 3t^2/(2r^2)$. Substituting
$r^2 = 3\sigma^2/4$:
\begin{equation}\label{eq:rate-computation}
\frac{3}{2r^2}
  = \frac{3}{2 \cdot 3\sigma^2/4}
  = \frac{2}{\sigma^2},
\end{equation}
so $\Lambda^*(t) \leq 2t^2/\sigma^2$, yielding~\eqref{eq:lower-bound}.
The condition $t \leq r/2 = \sigma\sqrt{3}/4$ is satisfied since
$t \leq \sigma/(4\sqrt{n}) \leq \sigma/4 < \sigma\sqrt{3}/4$.
\end{proof}

Note that this result establishes the \emph{rate} $\exp(-cnt^2)$ as
optimal. It does not claim that the numerical constants in our upper
bounds are sharp, and closing the gap between the constants remains open
(Section~\ref{sec:discussion}).

\begin{remark}[Rate versus constants]\label{rem:sharpness-rate}
The upper and lower bounds both decay as $\exp(-cnt^2)$, so the
sub-Gaussian exponent is the correct scaling. The numerical constants
differ, however. The lower bound has $2/\sigma^2$ in the exponent
while Corollary~\ref{thm:bernstein-cor} gives $1/(8\bar\sigma_n^2)$.
Two effects account for this. The covering argument introduces a
factor of~$4$ through the $t \to t/2$ halving and the union bound
over the net. Additionally, $\bar\sigma_n^2$ is a worst-case
parameter, whereas the construction uses a specific distribution with
variance $\sigma^2/4$. The dimension-free bound
(Theorem~\ref{thm:dimfree}) avoids the covering loss entirely and
achieves the tighter constant $1/(2\bar\sigma_n^2)$, which suggests
that much of the gap comes from the covering technique rather than
from any fundamental limitation.
\end{remark}

\begin{remark}[Extension to $d > 1$]\label{rem:sharpness-d}
The construction generalises to~$\R^d$. Take
$\Theta_1 \subset \R^d$ convex and compact, let
$a_0 \in \partial\Theta_1$ be a boundary point with outward unit
normal~$\nu$, and replace the one-dimensional uniforms by
$d$-dimensional measures on balls of radius $r$ centred at
$\mu \in \Theta_1$. Projecting onto $\nu$ recovers a sub-Gaussian
lower bound in every dimension.
\end{remark}


\section{Discussion and open problems}\label{sec:discussion}

Several directions remain open. Our results assume
$|X_i| \leq M$. Extending to a sub-Gaussian condition
$\Ehat[\exp(\lambda\ip{p}{X_i})] \leq \exp(C\lambda^2)$ requires
transferring MGF estimates through the iterated conditioning of
Definition~\ref{def:independence}, which amounts to a truncation
device that preserves conditional sub-Gaussianity. On the CLT side,
Peng's theorem~\cite{Peng2019} gives convergence of
$\sqrt{n}\,\Xbar$ to a $G$-normal distribution, and a multivariate
Berry--Esseen rate (extending the scalar $n^{-1/2}\log n$ rate
of~\cite{FangPengShaoSong2019}) appears within reach via our
dimension-free technique combined with Stein-type couplings. Finally,
while Theorem~\ref{thm:sharpness} pins down the rate
$\exp(-cnt^2)$, the optimal constant $c^*$ remains undetermined.


\end{document}